\documentclass[11pt]{article}

\usepackage{latexsym}
\usepackage{amssymb}
\usepackage{amsthm}

\newtheorem{thm}{Theorem}
\newtheorem{cor}{Corollary}
\newtheorem{prop}{Proposition}

\newtheorem{example}{Example}

\begin{document}
{
\begin{center}
{\Large\bf
A refined polar decomposition for $J$-unitary operators.}
\end{center}
\begin{center}
{\bf S.M. Zagorodnyuk}
\end{center}

\section{Introduction.}

Last years $J$-symmetric, $J$-skew-symmetric and $J$-unitary operators attracted still more and more attention of researches,
see, e.g. ~\cite{cit_100_GP}, \cite{cit_200_GP2}, \cite{cit_300_Z}, \cite{cit_400_LZ} and references therein.
Recall that a conjugation $J$ in a Hilbert space $H$ is an antilinear operator on $H$ such that $J^2 x = x$, $x\in H$,
and
$$ (Jx,Jy)_H = (y,x)_H,\qquad x,y\in H. $$
The conjugation $J$ generates the following bilinear form:
\begin{equation}
\label{f1_1}
[x,y]_J := (x,Jy)_H,\qquad x,y\in H.
\end{equation}
A linear operator $A$ in $H$ is said to be $J$-symmetric ($J$-skew-symmetric) if
\begin{equation}
\label{f1_2}
[Ax,y]_J = [x,Ay]_J,\qquad x,y\in D(A),
\end{equation}
or, respectively,
\begin{equation}
\label{f1_3}
[Ax,y]_J = -[x,Ay]_J,\qquad x,y\in D(A).
\end{equation}
A linear operator $A$ in $H$ is said to be $J$-isometric if
\begin{equation}
\label{f1_3_1}
[Ax,Ay]_J = [x,y]_J,\qquad x,y\in D(A).
\end{equation}
A linear operator $A$ in $H$ is called $J$-self-adjoint ($J$-skew-self-adjoint, or $J$-unitary) if
\begin{equation}
\label{f1_4}
A = JA^*J,
\end{equation}
or
\begin{equation}
\label{f1_5}
A = -JA^*J,
\end{equation}
or
\begin{equation}
\label{f1_6}
A^{-1} = JA^*J,
\end{equation}
respectively.

A refined polar decomposition for complex symmetric operators was obtained by Garcia and Putinar in~\cite{cit_200_GP2}.
Using  the technique of Garcia and Putinar, an analog for complex skew-symmetric operators was obtained by Li and Zhou in~\cite[Lemma 2.3]{cit_400_LZ}.
In this paper, we shall characterize the components of the polar decomposition for an arbitrary $J$-unitary operator.
This characterization has a quite different structure as that of the above-mentioned decompositions.
For the case of a bounded $J$-unitary operator, a similar decomposition was obtained in~\cite[Theorem 3.2]{cit_300_Z}.
However, in the unbounded case we can not use arguments from~\cite{cit_300_Z}.

\noindent
A linear operator $A$ in a Hilbert space $H$ is said to be \textit{$J$-imaginary} (\textit{$J$-real}) if $f\in D(A)$ implies $Jf\in D(A)$ and
$AJf = -JAf$ (respectively $AJf = JAf$), where $J$ is a conjugation on $H$.
We shall answer a question of the existence of $J$-imaginary self-adjoint extensions of $J$-imaginary symmetric operators.
This subject is similar to the study of
$J$-real self-adjoint extensions of $J$-real symmetric operators, see~\cite{cit_530_Stone}.
However, we can not state that a $J$-imaginary symmetric operator has equal defect numbers.
Nevertheless, it is shown that a $J$-imaginary self-adjoint extension of a $J$-imaginary symmetric operator exists
in a possibly larger Hilbert space.

\textbf{Notations.}
As usual, we denote by $\mathbb{R}, \mathbb{C}, \mathbb{N}, \mathbb{Z}, \mathbb{Z}_+$,
the sets of real numbers, complex numbers, positive integers, integers and non-negative integers,
respectively; $\mathbb{R}_e = \mathbb{C}\backslash \mathbb{R}$. By $I_d$ we denote the unit matrix of order $d$; $d\in \mathbb{N}$.
By $\mathfrak{B}(S)$ we mean a set of all Borel subsets of $S\subseteq \mathbb{C}$.
If H is a Hilbert space then $(\cdot,\cdot)_H$ and $\| \cdot \|_H$ mean
the scalar product and the norm in $H$, respectively.
Indices may be omitted in obvious cases.
For a linear operator $A$ in $H$, we denote by $D(A)$
its  domain, by $R(A)$ its range, and $A^*$ means the adjoint operator
if it exists. If $A$ is invertible then $A^{-1}$ means its
inverse. $\overline{A}$ means the closure of the operator, if the
operator is closable.
If $A=A^*$, then $\mathcal{R}_z(A) := (A-zE_H)^{-1}$, $z\in \mathbb{R}_e$.
For a set $M\subseteq H$
we denote by $\overline{M}$ the closure of $M$ in the norm of $H$.
By $E_H$ we denote the identity operator in $H$, i.e. $E_H x = x$,
$x\in H$. In obvious cases we may omit the index $H$.

\section{Properties of $J$-unitary operators.}

The following proposition accumulates some basic properties of $J$-unitary operators.

\begin{prop}
\label{p2_1}
Let $J$ be a conjugation on a Hilbert space $H$, and  $A$ be a $J$-unitary operator in $H$.
Then the following statements are true:

\begin{itemize}

\item[(i)] $A$ is closed;

\item[(ii)] $A^{-1}$ is $J$-unitary;

\item[(iii)] $A^*$ is $J$-unitary;


\item[(iv)] $A^*A$ is $J$-unitary;

\item[(v)] If $A$ is bounded, then $D(A)=R(A)=H$.

\end{itemize}
\end{prop}

\begin{proof}
Let $A$ be a $J$-unitary operator in a Hilbert space $H$.
By~\cite[Proposition 2.8]{cit_300_Z} we may write: $A^{-1} = J A^* J = (JAJ)^*$, and therefore $A^{-1}$ and $A$ are closed.
Moreover, by~\cite[Proposition 2.10]{cit_300_Z} we get
$$ ( A^{-1} )^{-1} = (J A^* J)^{-1} = J (A^*)^{-1} J = J (A^{-1})^* J, $$
and therefore $A^{-1}$ is $J$-unitary.
Since
$$ (A^*)^{-1} = (A^{-1})^* = JAJ, $$
then $A^*$ is $J$-unitary.
Set $G = A^* A$. The operator $G$ is non-negative and we may write:
$$ JG^* J = JA^* AJ = JA^* J JAJ = A^{-1} (A^*)^{-1} = (A^* A)^{-1}. $$
Thus, $G$ is $J$-unitary.

If $A$ is bounded, then $A^{-1} = J A^* J$ is bounded, and the closeness of $A$ and $A^{-1}$ implies \textit{(v)}.

\end{proof}

Now we can obtain a refined polar decomposition for a $J$-unitary operator.

\begin{thm}
\label{t2_1}
Let $J$ be a conjugation on a Hilbert space $H$. Then the following assertions hold:

\begin{itemize}

\item[1)]
If $A$ is a $J$-unitary operator in $H$ then
\begin{equation}
\label{f2_1}
A = UB,
\end{equation}
where $U$ is a unitary $J$-real operator, and $B$ is a non-negative self-adjoint $J$-unitary operator;

\item[2)]
If an operator $A$ in $H$ admits a representation~(\ref{f2_1}) with
a unitary $J$-real operator $U$, and a non-negative self-adjoint $J$-unitary operator $B$,
then $A$ is $J$-unitary.
\end{itemize}
\end{thm}

\begin{proof}
Let $A$ be a $J$-unitary operator in a Hilbert space $H$.
Set $G = A^*A$, and let $E_G(\delta)$, $\delta\in \mathfrak{B}(\mathbb{R})$ be the spectral measure of $G$.
By Proposition~\ref{p2_1} we conclude that $G$ is $J$-unitary.
Let us check that \textit{$E(\delta) := J E_G(\delta) J$, $\delta\in \mathfrak{B}(\mathbb{R})$ is the spectral measure of $G^{-1}$}.
In fact, $E(\delta)$ satisfies conditions $E^2 = E$, $E^* = E$, therefore $E(\delta)$ is a projection operator. The strong
$\sigma$-additivity of $E$ follows from the strong $\sigma$-additivity of $E_G$ and the continuity of $J$.
Moreover, $E(\mathbb{R}) = J E_G(\mathbb{R}) J = E_H$.
Thus, $E$ is a spectral measure. Denote by $T$ the corresponding to $E$ self-adjoint operator in $H$.
Observe that
$$ \mathcal{R}_z(G^{-1}) = (G^{-1} - zE_H)^{-1} = (JG^*J - J\overline{z}J)^{-1} $$
$$ = J (G^* - \overline{z} E_H)^{-1} J = J \mathcal{R}_z^*(G) J,\qquad z\in \mathbb{R}_e. $$
For arbitrary $f,g\in H$, $z\in \mathbb{R}_e$, we may write:
$$ ( \mathcal{R}_z(G^{-1}) f,g ) = (J \mathcal{R}_z^*(G) J f,g) = (\mathcal{R}_z(G) J g, Jf) $$
$$ = \int \frac{1}{s-z} d(E_G(s) Jg, Jf) = \int \frac{1}{s-z} d(E(s) f, g) =  ( \mathcal{R}_z(T) f,g ). $$
Therefore $T=G^{-1}$.
Notice that
$$ \mathcal{R}_z (J|A|J) = (J|A|J - zE_H)^{-1} = (J (|A| - \overline{z}E_H) J)^{-1}=
J \mathcal{R}_z^* (|A|) J,\quad z\in \mathbb{R}_e. $$
For arbitrary $f,g\in H$, $z\in \mathbb{R}_e$, we may write:
$$ (\mathcal{R}_z (J|A|J) f,g) = (J \mathcal{R}_z^* (|A|) J f,g) =
(\mathcal{R}_z (|A|) Jg,Jf) $$
$$ = \int \frac{1}{\sqrt{s}-z} d(E_G Jg,Jf) = \int \frac{1}{\sqrt{s}-z} d(E f,g) =
\left( \int \frac{1}{\sqrt{s}-z} dE f, g \right) $$
$$ = (\mathcal{R}_z (\sqrt{G^{-1}}) f,g). $$
Therefore
\begin{equation}
\label{f2_1_5}
J |A| J = \sqrt{G^{-1}}.
\end{equation}
Let us check that
\begin{equation}
\label{f2_1_7}
\sqrt{G^{-1}} = \left( \sqrt{G} \right)^{-1}.
\end{equation}
In fact, using the change of a variable:
$$ \lambda = \pi(u) = \left\{ \begin{array}{cc} \sqrt{u}, & u\geq 0\\
u, & u<0 \end{array}\right., $$
for the spectral measure $E_G$ (see, e.g., \cite{cit_550_BS})
we obtain the spectral measure $E_{\sqrt{G}}$ of $\sqrt{G}$, and
we may write:
$$ \left( \sqrt{G} \right)^{-1} = \left( \int \sqrt{u} dE_G \right)^{-1}
= \left( \int \lambda dE_{\sqrt G} \right)^{-1} $$
\begin{equation}
\label{f2_1_9}
= \int \frac{1}{\lambda} dE_{\sqrt{G}} =
\int \frac{1}{\sqrt{u}} dE_G(u).
\end{equation}
On the other hand, using the change of a variable:
$$ \lambda = \widehat\pi(s) = \left\{ \begin{array}{cc} \frac{1}{s}, & s>0\\
s, & s\leq 0 \end{array}\right., $$
for the spectral measure $E$ of $G^{-1}$, we obtain the spectral
measure $E_G$, and
we may write
\begin{equation}
\label{f2_1_11}
\sqrt{G^{-1}} = \int \sqrt{s}  dE(s) = \int \frac{1}{ \sqrt{\lambda} } dE_G (\lambda).
\end{equation}
By~(\ref{f2_1_9}),(\ref{f2_1_11}) we conclude that relation~(\ref{f2_1_7}) holds.

\noindent
By~(\ref{f2_1_5}),(\ref{f2_1_7}) we obtain that
$J |A| J = |A|^{-1}$.
Thus, $B := |A|$ is $J$-unitary.

Consider the polar decomposition for $A$: $A = UB$, where $U$ is a unitary operator in $H$
(since $\overline{R(A)} = \overline{R(B)} = H$). Then
$A^* = B^* U^*$ (since $U$ is bounded on $H$) and
$$ UB^{-1} = (BU^{-1} )^{-1} = ( A^* )^{-1} = JAJ = JUJ JBJ $$
$$ = JUJ B^{-1}. $$
Therefore
$$ U h = JUJ h,\qquad h\in D(B). $$
By the continuity we conclude that $U$ is $J$-real.

Let us check assertion~2) of the theorem. For the operator $A$ in this case we may write:
\begin{equation}
\label{f2_5}
JAJ = JUJ JBJ = U B^{-1},
\end{equation}
\begin{equation}
\label{f2_6}
A^{-1} = B^{-1} U^{-1}.
\end{equation}
Since $U$ is bounded on $H$, we may write:
$$ J A^* J = (JAJ)^* = (UB^{-1})^* = (B^{-1})^* U^* = B^{-1} U^{-1} = A^{-1}. $$
Therefore $A$ is $J$-unitary.

\end{proof}

\begin{cor}
\label{c2_1}
Let $J$ be a conjugation on a Hilbert space $H$, and $A$ be a $J$-unitary operator in $H$.
Then operators $A^* A$ and $A A^*$ are unitarily equivalent.
\end{cor}
\begin{proof}
In the notations of Theorem~\ref{t2_1} we may write: $A^* A = B^2$, and, since $U$ is bounded,
$A A^* = UB (UB)^* = UB B^* U^* = UB B U^{-1} = U A^*A U^{-1}$.
Observe that we only used that $U$ is unitary in the polar decomposition of $A$.
\end{proof}
As it was noticed in~\cite{cit_200_GP2}, for the unilateral shift $A$ the operators $A^* A$ and $A A^*$ are not unitarily equivalent.
Thus, the unilateral shift is not $J$-unitary.

\begin{example} (An unbounded $J$-unitary operator)
\label{e2_1}
Let
$$ A_0 := A_0(\beta) := \left( \begin{array}{cc} 0 & \beta i\\
-\beta i & 0\end{array}
\right),\qquad \beta\in (-1,1). $$
Observe that
$$ (A_0(\beta) \pm I_2)^{-1} = \frac{1}{1-\beta^2}
\left( \begin{array}{cc} \pm 1 & -\beta i\\
\beta i & \pm 1\end{array}
\right). $$
Let $H = \bigoplus\limits_{k=1}^\infty H_k$, where $H_k = \mathbb{C}^2$ is the space of $2$-dimensional complex vectors,
and
$A = \bigoplus\limits_{k=1}^\infty A_0 \left( 1 - \frac{1}{k}  \right)$.
For an element of $H$ of the form $h=(h_j)_{j=1}^\infty$, $h_j = \left( \begin{array}{cc} h_{j,1} \\
h_{j,2} \end{array} \right)\in H_j$, we set
$J h = (\mathcal{J} h_j)_{j=1}^\infty$, where $\mathcal{J} h_j = \left( \begin{array}{cc} \overline{h_{j,1}} \\
\overline{h_{j,2}} \end{array} \right)$. Observe that $J$ is a conjugation on $H$.
It is straightforward to check that
$A$ is a bounded self-adjoint, $J$-skew-self-adjoint operator on $H$, and there exist $(E_H \pm A)^{-1}$.
Let $e_{k,1}$ be an element of $H$ of the form $(h_j)_{j=1}^\infty$, $h_j\in H_j$, where $h_j= \delta_{j,k}
\left( \begin{array}{cc} 1 \\
0 \end{array} \right)$; $k\in \mathbb{N}$.
Observe that
$$ (E_H+A)^{-1} e_{k,1} = \left(
\delta_{j,k}
\frac{1}{ 1 - (1-\frac{1}{k})^2 }
\left(
\begin{array}{cc} 1\\
\left( 1-\frac{1}{k} \right) i\end{array}
\right)
\right)_{j=1}^\infty. $$
Since
$$ \left( (E_H + A)^{-1} e_{k,1}, e_{k,1}  \right)_H = \frac{1}{1 - \left( 1 - \frac{1}{k}   \right)^2} \rightarrow \infty, $$
as $k\rightarrow\infty$, then $(E_H + A)^{-1}$ is unbounded.
Consider the following operator:
\begin{equation}
\label{f2_15}
V = (A + E_H) (A - E_H)^{-1} = E_H + 2 (A - E_H)^{-1}.
\end{equation}
Transformation~(\ref{f2_15}), which connects some $J$-skew-symmetric and $J$-isometric operators,
was studied by~Kamerina
in~\cite{cit_1700_K}.
Observe that
$$ V = \int \frac{\lambda+1}{\lambda-1} dE_A(\lambda), $$
where $E_A(\lambda)$ is the spectral measure of $A$. Thus, $V$ is self-adjoint, and we may write:
$$ JV^*J = JVJ = E_H + 2 J(A - E_H)^{-1}J = E_H - 2 (A + E_H)^{-1} $$
$$ = (A - E_H) (A + E_H)^{-1}  = V^{-1}. $$
Thus, $V$ is a $J$-unitary operator.
Therefore $V^{-1}$ is an unbounded $J$-unitary operator.
\end{example}

Unitary $J$-real operators, which appear in the refined polar decomposition~(\ref{f2_1}), also play an important role
in the question of an extension of $J$-imaginary symmetric operators to $J$-imaginary self-adjoint operators.

\begin{thm}
\label{t2_2}
Let $J$ be a conjugation on a Hilbert space $H$. Let $A$ be a closed $J$-imaginary symmetric operator in $H$,
$\overline{D(A)}=H$. Then
there exists a $J$-imaginary self-adjoint operator $\widetilde A\supseteq A$ in a Hilbert space
$\widetilde H\supseteq H$
(with an extension of $J$).
If the defect numbers of $A$ are equal, then there exists a $J$-imaginary self-adjoint operator $\widehat A\supseteq A$
in $H$.
\end{thm}
\begin{proof}
At first, suppose that the defect numbers of $A$ are equal.
Consider Cayley's transformation of $A$:
$$ U_z = U_z(A) = ( A - \overline{z} E_H )(A-zE_H)^{-1} = E_H + (z-\overline{z}) (A-zE_H)^{-1},\quad z\in \mathbb{C}. $$
Observe that
$$ J \mathcal{M}_z(A) = \mathcal{M}_{-\overline{z}}(A),\qquad z\in \mathbb{C}, $$
where $\mathcal{M}_\lambda(A) := (A-\lambda E_H) D(A)$, $\lambda\in \mathbb{C}$.
In particular, we see that
\begin{equation}
\label{f2_17}
J \mathcal{M}_{\pm i}(A) = \mathcal{M}_{\pm i}(A).
\end{equation}
Then
\begin{equation}
\label{f2_17_1}
J \mathcal{N}_{\pm i}(A) = \mathcal{N}_{\pm i}(A),
\end{equation}
where $\mathcal{N}_\lambda(A) := H\ominus \mathcal{M}_\lambda(A)$, $\lambda\in \mathbb{C}$.

Let $W$ be an arbitrary linear $J$-real isometric operator, which maps $\mathcal{N}_{i}(A)$ onto $\mathcal{N}_{-i}(A)$.
In particular, if $\mathfrak{A}_\pm = \{ f_k^\pm \}_{k=0}^\tau$, $0\leq\tau\leq +\infty$, is an orthonormal
basis in $\mathcal{N}_{\pm i}(A)$, corresponding to $J$ (i.e. $J f_k^\pm = f_k^\pm$), then we may set
$$ W \sum_{k=0}^\tau \alpha_k f_k^+ = \sum_{k=0}^\tau \alpha_k f_k^-,\qquad \alpha_k\in \mathbb{C}. $$
Then
$V := U_i\oplus W$ is a $J$-real unitary operator in $H$.
Observe that $\widetilde A := iE_H + 2i (V-E_H)^{-1}\supseteq A$, is self-adjoint and $J$-imaginary.

In the case of unequal defect numbers, we may consider an operator
$\mathcal{A} := A\oplus (-A)$ in a Hilbert space $\mathcal{H} := H\oplus H$ with a conjugation
$\mathcal{J} = J\oplus J$. The operator $\mathcal{A}$ is closed symmetric, $\mathcal{J}$-imaginary,
$\overline{ D(\mathcal{A}) } = \mathcal{H}$,
and
it has equal defect numbers. Thus, we may apply to $\mathcal{A}$ the already proved part.
\end{proof}

\begin{example} (A $J$-imaginary symmetric operator)
\label{e2_2}
Consider the usual space $H=l_2$ of square summable sequences of complex numbers $h = \left( \begin{array}{cccc} h_0\\
h_1\\
h_2\\
\vdots\end{array}\right)$.
A conjugation $J$ will be the following one: $Jh = \left( \begin{array}{cccc} \overline{h_0}\\
\overline{h_1}\\
\overline{h_2}\\
\vdots\end{array}\right)$.
An operator $A$ we shall define on a set of all finite vectors $\mathcal{F}$
(i.e. vectors which components are zeros except for a finite number)
by the following matrix multiplication:
$$ A h =
i \left( \begin{array}{ccccc}
0 & \alpha_0 & 0 & 0 & \ldots\\
-\alpha_0 & 0 & \alpha_1 & 0 & \ldots\\
0 & -\alpha_1 & 0 & \alpha_2 & \ldots\\
\vdots & \vdots & \vdots & \vdots & \ddots\end{array}
\right) h. $$
It is straightforward to check that $A$ is symmetric and $J$-imaginary.
Observe that $\overline{A}$ is $J$-imaginary, as well.
Applying Theorem~\ref{t2_2} to the operator $\overline{A}$ we conclude that
the operator $A$ has a self-adjoint $J$-imaginary extension in a Hilbert space
$\widetilde H\supseteq H$.

\end{example}

\begin{center}
{\large\bf A refined polar decomposition for $J$-unitary operators}
\end{center}
\begin{center}
{\bf S.M. Zagorodnyuk}
\end{center}

In this paper, we shall characterize the components of the polar decomposition for an arbitrary $J$-unitary operator
in a Hilbert space.
This characterization has a quite different structure as that for complex symmetric and complex skew-symmetric
operators.
It is also shown that for a $J$-imaginary closed symmetric operator in a Hilbert space
there exists a $J$-imaginary self-adjoint extension
in a possibly larger Hilbert space
(a linear operator $A$ in a Hilbert space $H$ is said to be $J$-imaginary if $f\in D(A)$ implies $Jf\in D(A)$ and
$AJf = -JAf$, where $J$ is a conjugation on $H$).

}


\end{document}